\documentclass[11pt]{amsart}
\usepackage[small,width=\linewidth]{caption}

\def\cput(#1,#2)#3{\put(#1,#2){\hbox to 0pt{\hss{#3}\hss}}}
\def\lput(#1,#2)#3{\put(#1,#2){\hbox to 0pt{\hss{#3}}}}

\title[\resizebox{5.3in}{!}{An algebraic approach to polynomial reproduction of Hermite subdivision schemes}]{An algebraic approach to polynomial reproduction of Hermite subdivision schemes}

\author{Costanza Conti, Svenja H\"uning}

\address{Costanza Conti (costanza.conti@unifi.it). DIEF, Universit\`{a} di Firenze. Viale Morgagni 40/44, 50134 Firenze, Italy. Svenja H\"uning (huening@tugraz.at). Institut f.\ Geometrie, TU Graz. Kopernikusgasse 24, 8010 Graz, Austria.}

\usepackage{amsmath,amsfonts,amssymb,amsthm,overpic,relsize}
\usepackage[square,numbers]{natbib}
\usepackage{diagbox}
\usepackage[a4paper]{geometry}
\usepackage{float}

\theoremstyle{plain}
\newtheorem{thm}{Theorem}
\newtheorem{lem}[thm]{Lemma}

\newtheorem{prop}[thm]{Proposition}
\newtheorem{conj}[thm]{Conjecture}

\theoremstyle{definition}
\newtheorem{Def}[thm]{Definition}
\newtheorem{ex}[thm]{Example}

\theoremstyle{remark}
\newtheorem{Rem}[thm]{Remark}

\def \<{\langle}
\def \>{\rangle}

\begin{document}
\long\def\nix#1{}

\maketitle
\begin{abstract}
We present an accurate investigation of the algebraic conditions that the symbols of a
univariate, binary, Hermite subdivision scheme have to fulfil in order to reproduce polynomials. These conditions are sufficient for the scheme  to
satisfy the so called spectral condition. The latter requires the existence of particular polynomial eigenvalues of the stationary counterpart of the Hermite scheme.
In accordance with the known Hermite schemes, we here consider the case of a Hermite scheme dealing with function values, first and second derivatives.
Several examples of application of the proposed algebraic conditions are given in both the primal and the dual situation.
\end{abstract}

\section{Introduction}
Hermite subdivision are linear operators which act on vector data but with the particular understanding that these vectors represent function values and consecutive derivatives up to a certain order
(see \cite{conti17, merrien5, merrien1}, for example). They find application in different contexts ranging from interpolation/approximation \cite{han2, han1,RMS16} to the construction of Hermite-type multiwavelets  \cite{cotronei17, Vonesch}, up to biomedical imaging \cite{romani, uhlmann}.

 In this paper, we study the capability of a Hermite subdivision scheme to
reproduce polynomials in the sense that, for initial data sampled from a polynomial function, the scheme yields the same polynomial and its derivatives in the limit.
The polynomial reproduction guarantees that the scheme satisfies the so called spectral condition that allows factorization of the subdivision operator which in the end leads to convergence results \cite{merrien3, merrien4}.
Moreover, the polynomial reproduction is strictly connected  to the approximation order of the scheme \cite{CYR16}, even in the Hermite case \cite{byeongseon}.

Our study, of purely algebraic nature,  provides algebraic conditions on the subdivision symbol and its derivatives for computing the exact degree of polynomial reproduction and also for determining the associated
correct parametrization. In this respect, it generalizes the work done in \cite{hormann} where the polynomial reproduction of a \emph{scalar} subdivision scheme is considered in full generality.

The case we study here, is the case of a Hermite scheme dealing with function values and first derivatives or  function values, first and second derivatives. Essentially, these are the known existing Hermite schemes. However, our algebraic approach can be extended to a general situation where function values and derivatives of order higher than $2$ are considered.

 The main result of our paper, Theorem \ref{main_result_2,3}, states that for a Hermite scheme associated with a finitely supported matrix mask $\mathcal{A}=\lbrace A_l\in \mathbb{R}^{d \times d}, l \in \mathbb{Z}\rbrace$, $d=2,3$, having the symbol $\textbf{A}(z)=\sum_{l \in \mathbb{Z}} A_l z^l$,
the polynomial reproduction of order $m$ is equivalent to specific properties of the \emph{symbol} and its derivatives evaluated at $1$ and $-1$, i.e., 
	\begin{align*}
	\textbf{A}(-1)\textbf{e}_{1,d}=\textbf{0}_{d},\qquad\textbf{A}^{(k)}(-1) \textbf{e}_{1,d}+\sum_{s=2}^{d}\Big(\sum_{\ell =s-1}^k \alpha^{d}_{k,\ell}\cdot \textbf{A}^{(k-\ell)}(-1) \textbf{e}_{s,d}\Big)=\textbf{0}_{d},\ k=1,\dots ,m,\\
	\textbf{A}(1)\textbf{e}_{1,d}=2\textbf{e}_{1,d} , \qquad \textbf{A}^{(k)}(1) \textbf{e}_{1,d}+\sum_{s=2}^{d}\Big(\sum_{\ell =s-1}^k \tilde{\alpha}^{d}_{k,\ell}\cdot \textbf{A}^{(k-\ell)}(1) \textbf{e}_{s,d}\Big)=\textbf{q}_{d},\ k=1,\dots ,m.
	\end{align*}
with $\alpha^{d}_{k,\ell},\ \tilde{\alpha}^{d}_{k,\ell},\ k=1,\dots ,m$  suitable sets of real coefficients, for $d=2,3$.
In the formula above $ \textbf{e}_{s,d}$ denotes the $s$-th canonical vector of $\mathbb{R}^{d}$, $\textbf{0}_{d}$ is the zero vector of $\mathbb{R}^{d}$ while $\textbf{q}_{d}$ is a $d$-dimensional vector that depends on the parametrization of the scheme. Indeed, the entries of $\textbf{q}_{d}$ are defined by means of the polynomials 
$$\prod_{r=0}^{k-1}(\tau-r),\quad k=1,\dots,m,\qquad \hbox{where \(\tau\) is the parametrization of the scheme}.$$
For $m = 1$, the importance of the previous conditions is that it allows us to identify the correct parametrization that guarantees at least the reproduction of linear polynomials. The parametrization determines the grid points to which the newly computed values are attached at each step of subdivision recursion to ensure the higher degree of polynomial reproduction of a scheme.
The dependence on the parametrization is not surprising since, even in the scalar situation, the parametrization plays an important role to guarantee reproduction properties of a subdivision scheme \cite{hormann}.

 A significant application of the derived algebraic conditions on the subdivision symbols is the construction of new Hermite subdivision schemes with specific reproduction
properties in both the primal and the dual situation. It is worth mentioning that, with this respect, this work is related to the recent paper \cite{byeongseon} where the construction of Hermite subdivision schemes reproducing polynomials is in fact considered also by the help of an algebraic approach.

\medskip This paper is organised as follows: in Section 2  we fix the notation and recall some useful background while in Section 3 we present and analyse some auxiliary polynomials that are crucial for our analysis. The core of the paper is Section 4 where the algebraic conditions for polynomial reproduction are stated and proved for $d=2$. In Section 5 we apply the mentioned algebraic conditions to a known dual existing scheme and for deriving a new primal Hermite scheme. The generalization to the case $d=3$ is then considered in Section 6.


\section{Notation and background}
To fix the notation used in the paper, we recall that a (univariate) Hermite subdivision operator
\(H_{\mathcal{A}}\),  based on the \emph{matrix mask}  \(\mathcal{A}=\lbrace A_l\in \mathbb{R}^{d \times d}, l \in \mathbb{Z}\rbrace\), $d\geqslant 2$, acts  on a sequence of Hermite data $f_{n}=\{\textbf{f}_{n}(j),\ j\in \mathbb{Z}\}\) as
	\begin{align} \label{subdivision_rule}
	\textbf{D}^{n+1}\textbf{f}_{n+1}(i)=\sum_{j \in \mathbb{Z}} A_{i-2j}\textbf{D}^n\textbf{f}_{n}(j) \quad \forall~ i\in \mathbb{Z},\quad  n\geqslant 0,
	\end{align}
where \( \textbf{D}=diag(1,\ \frac12,\ \dots, \frac{1}{2^d})\).
The Hermite subdivision scheme, still denoted by \(H_{\mathcal{A}}\), is the repeated application of \(H_{\mathcal{A}}\) when starting with an Hermite-type initial vector sequence composed of function and derivative values. 
We associate to \(H_{\mathcal{A}}\) the \emph{matrix symbol} 
$\displaystyle{\textbf{A}(z)=\sum_{l \in \mathbb{Z}} A_l z^l}$
and \emph{sub-symbols} $\displaystyle{\textbf{A}_e(z)=\sum_{l \in \mathbb{Z}} A_{2l} z^{2l}}$, $\displaystyle{\textbf{A}_o(z)=\sum_{l \in \mathbb{Z}} A_{2l+1} z^{2l+1}},$
which are related by the equation
$
\displaystyle{\textbf{A}(z)=\textbf{A}_e(z)+\textbf{A}_o(z)}.
$
Their derivatives are 
	\begin{align*}
		\textbf{A}^{(k)}(z):= \sum_{l \in \mathbb{Z}}\prod_{r=0}^{k-1}(l-r)A_l z^{l-k},
	\end{align*}
	and, respectively,
		\begin{align*}
		\textbf{A}_e^{(k)}(z):= \sum_{l \in \mathbb{Z}} \prod_{r=0}^{k-1}(2l-r)A_{2l} z^{2l-k},\quad \textbf{A}_o^{(k)}(z):= \sum_{l \in \mathbb{Z}}\prod_{r=0}^{k-1}(2l+1-r)A_{2l+1} z^{2l+1-k}.
\end{align*} 

 \smallskip In this paper we are interested in both \emph{primal }and  \emph{dual} Hermite schemes. From a geometric point of view, primal Hermite subdivision schemes are those that at each iteration retain or modify the given vectors and create a \lq new\rq\  vector in between two \lq old\rq\ ones.
Dual schemes, instead, discard all given vectors after creating two new ones in between any pair of them. This fact is algebraically connected with the choice of the parameter values $t_{i}^n,\ i\in \mathbb{Z}$, to which we attach the vectors generated by the Hermite scheme.  More precisely, the primal parametrization is such that $t_{i}^n=\frac{i}{2^n}$ while
the dual one is given by $t_{i}^n=\frac{i-\frac12}{2^n}$. Therefore, we consider in this paper the parametrization \(t_{i}^n=\frac{i+\tau}{2^n}\) which includes primal and dual cases. We simply say that $\tau$ is the \emph{parameterization of the scheme} (see \cite{conti1}, for example).

\smallskip We continue with the notion of \emph{reproduction} for Hermite schemes.
\begin{Def} \label{def_reproduction}
A Hermite subdivision scheme \(H_{\mathcal{A}}\) with parametrization $\tau$ \textit{reproduces a function} \(g \in C^d(\mathbb{R})\) if for any initial vector sequence \(f_{0}=\lbrace \textbf{f}_{0}(j)=\begin{bmatrix} g(j)\\ \vdots\\ g^{(d)}(j)\end{bmatrix}, j \in \mathbb{Z}\rbrace\) the sequence \(f_{n}=\lbrace \textbf{f}_{n}(j),\  j \in \mathbb{Z}\rbrace\) defined by (\ref{subdivision_rule}) is 
$\textbf{f}_{n}(j)=[g((j+\tau)/2^n) \cdots  g^{(d)}((j+\tau)/2^n)]^T$ for all $n \in \mathbb{N}$ and $j \in \mathbb{Z}$.
\end{Def}

\section{Analysis of auxiliary polynomials}
In this section we study the properties of a special class of polynomials which are involved in the algebraic properties we present. 
In case of a $d$-dimensional Hermite scheme, we need $d$ different classes of polynomials. Having defined the first class, the remaining $d-1$ classes are closely related to the first one. Here we consider the polynomials corresponding to the case $d=2$.


\subsection{Polynomials $q_k$}
We start by defining the polynomials $q_{k}\in \prod_k$, (\(\prod_k\) denotes the set of polynomials up to degree \(k\)), as
	\begin{align}  \label{def_poly_indep_i}
q_{0}(x):=1,\quad
	q_{k}(x):=\prod_{r=0}^{k-1}(2x-r), \quad k> 0.
	\end{align}
Obviously, we can write them in terms of the canonical base of $\prod_k$, so that 
	\begin{align*}
	q_{k}(-x)&=\sum_{n=0}^{k}\gamma_{n}^{k}x^n,\quad \hbox{for some coefficients}\quad \gamma_{n}^{k}\in \mathbb{R}.
	\end{align*}
	The reason why we expand $q_{k}(-x)$ instead of $q_{k}(x)$, will become clear later on.  By definition \(\gamma_{k}^{k}=(-1)^k2^k\), hence \(\gamma_{k}^{k}\neq 0,\ k\geqslant 0\) while \(\gamma_{0}^{k}=0\) for all \(k\geqslant 1\). For each \(i \in \mathbb{Z}\) we define  the polynomials 
	\begin{align} \label{def_polynomial_1}
	q_{0,i}(x):=1, \quad q_{k,i}(x):=q_{k}\Big(x+\frac{i}{2}\Big)=\prod_{r=0}^{k-1}(2x+i-r),\quad k> 0,
	\end{align}
which can also be written in terms of the canonical base as 
	\begin{align} \label{coefficients1}
	q_{k,i}(-x)=\sum_{n=0}^{k}\gamma_{n}^{k,i}x^n,\quad \hbox{for some coefficients}\quad  \gamma_{n}^{k,i}\in \mathbb{R}. 
	\end{align}
Obviously, \(q_{k,0}=q_k\) and \(\gamma_{n}^{k,0}=\gamma_{n}^{k},\ n=0,\dots,k\). In the next lemma we collect some relations between the coefficients of the polynomials \(q_{k,i}\).

\begin{lem} \label{recursion_gamma} Let \(i\in \mathbb{Z}\).
For all \(k\geqslant 1\) the coefficients of the polynomials \(q_{k,i}\) as presented in (\ref{coefficients1}) satisfy 
	\begin{align*}
		\gamma_{0}^{k,i}&=(i-(k-1))\gamma_{0}^{k-1,i}, \\
		\gamma_{n}^{k,i}&=-2\gamma_{n-1}^{k-1,i}+(i-(k-1))\gamma_{n}^{k-1,i} \quad  n=1,\dots ,k-1,\\
		\gamma_{k}^{k,i}&=-2\gamma_{k-1}^{k-1,i}.
	\end{align*}
\end{lem}

\begin{proof}
For \(k=1\) the statement is true by definition of the polynomials and its coefficients. For \(k>1\) it follows by (\ref{def_polynomial_1}) that \(q_{k,i}(x)=q_{k-1,i}(x)(2x+i-(k-1))\). Thus using (\ref{coefficients1}) we obtain
	\begin{align*}
	\sum_{n=0}^{k}(-1)^{n}\gamma_{n}^{k,i}x^n&=\sum_{n=0}^{k-1}(-1)^{n}\gamma_{n}^{k-1,i}x^n(2x+i-(k-1)) \\
	&=2\sum_{n=0}^{k-1}(-1)^{n}\gamma_{n}^{k-1,i}x^{n+1}+\sum_{n=0}^{k-1}(-1)^{n}(i-(k-1))\gamma_{n}^{k-1,i}x^{n} \\
	&=2\sum_{n=1}^{k}(-1)^{n-1}\gamma_{n-1}^{k-1,i}x^{n}+\sum_{n=0}^{k-1}(-1)^{n}(i-(k-1))\gamma_{n}^{k-1,i}x^{n} \\
	&=2(-1)^{k-1}\gamma_{k-1}^{k-1,i}x^k+\sum_{n=1}^{k-1}(2(-1)^{n-1}\gamma_{n-1}^{k-1,i}+(-1)^n(i-(k-1))\gamma_{n}^{k-1,i})x^{n}\\  &\quad+ (i-(k-1))\gamma_{0}^{k-1,i}.
	\end{align*}
Comparison of the coefficients proves the lemma.
\end{proof}	

Next, we study the relation of the coefficients of the polynomials \(q_{k,i}\) (which do depend on \(i\in \mathbb{Z}\)) and those of the polynomials \(q_{k}\), defined in (\ref{def_poly_indep_i}) (which do not depend on \(i\in \mathbb{Z}\)).

\begin{lem} \label{lem_foralli}
For \(i\in \mathbb{Z}\) and $k\geqslant 0$, \(\gamma_{n}^{k,i}=\sum_{r=n}^{k}(-1)^{r+n}\gamma_{r}^{k} \binom{r}{n}(\frac{i}{2})^{r-n}\),  \(n=0,\dots,k\). 
Moreover, $\gamma_{k}^{k,i}\neq 0$.
\end{lem}

\begin{proof}
Let \(i\in \mathbb{Z}\). We have
\begin{align*}
	q_{k}\Big(x+\frac{i}{2}\Big)&=\sum_{r=0}^{k} \gamma_{r}^{k}\Big(-x-\frac{i}{2}\Big)^r =\sum_{r=0}^{k} (-1)^r\gamma_{r}^{k}\sum_{n=0}^{r} \binom{r}{n}\Big(\frac{i}{2}\Big)^{r-n}x^n\\ &=\sum_{n=0}^{k} \sum_{r=n}^{k}(-1)^r\gamma_{r}^{k}\binom{r}{n}\Big(\frac{i}{2}\Big)^{r-n}x^n.
\end{align*}
By (\ref{coefficients1}) a comparison of the coefficients leads to
	\begin{equation}\label{eq:gammai-gamma}
	\gamma_{n}^{k,i}=\sum_{r=n}^{k}(-1)^{r+n}\gamma_{r}^{k} \binom{r}{n}\Big(\frac{i}{2}\Big)^{r-n} \quad
	 \quad n=0,\dots ,k.
	\end{equation}
Finally, since $\gamma_{k}^{k,i}=\gamma_{k}^{k}$ and $\gamma_{k}^{k}\neq 0$ the proof is complete.
\end{proof}
\subsection{Polynomials $\tilde{q}_k$  and coefficients $\alpha^{1}_{k,\ell}$}

 We define a second class of polynomials which is closely related to the polynomials \(q_{k,i}\). First, we need to introduce the  coefficients \(\alpha^{1}_{k,\ell}\), for \(\ell=1,\dots,k\). They are defined in a recursive way as  

\begin{Def} \label{def_coefficients_c}
Let \(k\in \mathbb{N}\). We define the sequence of coefficients $\{\alpha^{1}_{k,\ell},\ \ell=1,\dots, k\}$
	\begin{align*}
\begin{array}{ll}
	\alpha^{1}_{k,1}&:=2k\\
	\alpha^{1}_{k,k-n+1}&:=(-1)^{k}2^{-n+1}\Big(n\gamma_{n}^{k}-\displaystyle{\sum_{j=1}^{k-n}(-1)^j\alpha^{1}_{k,j}\gamma_{n-1}^{k-j}}\Big) \quad n=k-1,\dots ,1.
\end{array}
	\end{align*}
\end{Def}

\smallskip \noindent In order to compute \(\alpha^{1}_{k,\ell}\) for some \(\ell \in \lbrace 2,...,k\rbrace\), we consider the second formula of Definition \ref{def_coefficients_c} for \(n=k-(\ell-1)\).
We illustrate the computation in an example.

\begin{ex}
Let \(k=3\). 
Computations show that \(\gamma_{0}^{1}=\gamma_{0}^{2}=0\), \(\gamma_{1}^{2}=2\) and \(\gamma_{1}^{3}=-4\), \(\gamma_{2}^{3}=-12\).
By Definition \ref{def_coefficients_c} it follows that
	\begin{align*}
\alpha^{1}_{3,1}=6,\quad	\alpha^{1}_{3,2}=-\frac{1}{2}\Big(2\gamma_{2}^{3}+6\gamma_{1}^{2}\Big)=6, \quad 	\alpha^{1}_{3,3}=-(\gamma_{1}^{3}+\alpha^{1}_{3,1}\gamma_{0}^{2}-\alpha^{1}_{3,2}\gamma_{0}^{1})=4.
	\end{align*}
For more explicit values of the coefficients, see Table \ref{table_c_kl}.
\end{ex}

Numerical computations give the values of the coefficients \(\alpha^{1}_{k,\ell},\ \ell=1,\dots,k\), defined in Definition  \ref{def_coefficients_c},  as shown in Table \ref{table_c_kl}.
Therefore, we conjecture what follows.
\begin{conj} The coefficients \(\alpha^{1}_{k,\ell}\), \(k \in \mathbb{N}\), can be computed directly by the rules
	\begin{align*}
	\alpha^{1}_{k, 1}&:=2k, \\
	\alpha^{1}_{k,\ell}&:=2(\ell+1)(\ell-1)!+(\ell-1)\sum_{i=\ell}^{k-2}\alpha^{1}_{i+1,\ell-1} \quad \ell=2,\dots, k-1, \\
	\alpha^{1}_{k,k}&:=2(k-1)!,
	\end{align*}
where we implicitly assume the convention  \(\sum_{i=\ell}^{k-2}\alpha^{1}_{i+1,\ell-1}=0\), if \(\ell=k-1\). So, \(\alpha^{1}_{k,k-1}=2k(k-2)!\).

	\begin{table} \label{table_c_kl}
	\centering \small
		\begin{tabular}{|l|ccccccc|}
		\hline
   		\diagbox{\(k\)}{\(\ell\)}
     	& 1& 2& 3& 4 & 5  & 6 & 7 \\
  		\hline
	 	1 & 2 & & & & &  &   \\
  		2 & 4 & 2 &  & &  &  &   \\
	    3 & 6 & 6 & 4 &   &  &  & \\
   		4 & 8 & 12 & 16 & 12 &  &  &   \\
  		5 & 10 & 20 & 40 & 60 & 48 &  &  \\
  		6 & 12 & 30 & 80 & 180 & 288 & 240 &   \\
  		7 & 14 & 42 & 140 & 420 & 1008 & 1680 & 1440  \\
  				\hline	
		\end{tabular}
	\caption{Values of the coefficients \(\alpha^{1}_{k,\ell},\ k=1,\dots,7,\ \ell=1,\dots,k\).}
	\label{table_c_kl}
	\end{table}
\end{conj}

\smallskip \noindent Based on the previous set of coefficients, for $i \in \mathbb{Z}$ and  $k\geqslant 0$ we define the polynomials $\tilde{q}_{k,i}\in \prod_{k-1}$ 
	\begin{align} \label{def_polynomial_2}
\tilde{q}_{0,i}:=0,\quad 
	\tilde{q}_{k,i}(x):=\sum_{n=1}^{k}(-1)^{n}\alpha^{1}_{k,n}q_{k-n,i}(x),\quad k> 0,\quad  i \in \mathbb{Z}.
\end{align}
As done before, we can write them in the form
	\begin{align}
	\tilde{q}_{k,i}(-x)&=\sum_{n=0}^{k-1}\tilde{\gamma}_{n}^{k,i}x^n,\quad \hbox{for some coefficients}\quad \tilde{\gamma}_{n}^{k,i}\in \mathbb{R}. \label{coefficients2}
	\end{align}

\begin{lem} \label{relation_coefficients2}
For \(k\in \mathbb{N}\)  and \(i \in \mathbb{Z}\) we have
\(\tilde{\gamma}_{k-1}^{k,i}=k\gamma_{k}^{k,i}\neq0\) and \(\tilde{\gamma}_{n}^{k,i}=\sum_{j=1}^{k-n}(-1)^{j}\alpha^{1}_{k,j}\gamma_{n}^{k-j,i}\) for \(n=0,\dots, k-1.\)
\end{lem}

\begin{proof}
By definition of the polynomial \(\tilde{q}_{k,i}\) in (\ref{def_polynomial_2}) and its representation in (\ref{coefficients2}), we obtain
\(\tilde{q}_{k,i}(-x)=\sum_{n=1}^{k}(-1)^n\alpha^{1}_{k,n}q_{k-n,i}(-x)\). Thus the coefficient \(\tilde{\gamma}_{k-1}^{k,i}\), which belongs to the \(x^{k-1}\) term, is given by \(\tilde{\gamma}_{k-1}^{k,i}=-\alpha^{1}_{k,1}\gamma_{k-1}^{k-1,i}=k\gamma_{k}^{k,i}\). Here the last equality follows by Lemma \ref{recursion_gamma}. By Lemma \ref{lem_foralli} \(\gamma_{k}^{k,i}\neq 0\). We continue by computing
	\begin{align*}
	\tilde{q}_{k,i}(-x)=\sum_{n=1}^{k}(-1)^{n}\alpha^{1}_{k,n} \sum_{j=0}^{k-n}\gamma_{j}^{k-n,i}x^j=\sum_{j=0}^{k-1}\Big( \sum_{n=1}^{k-j}(-1)^{n}\alpha^{1}_{k,n}\gamma_{j}^{k-n,i}\Big)x^j.
	\end{align*}
This proves the second part of the lemma.
\end{proof}
Next, we compare the coefficients of the polynomials ${q}_{k,i}$ and $\tilde{q}_{k,i}$ defined in (\ref{def_polynomial_1}) and (\ref{def_polynomial_2}), respectively.

\begin{prop} \label{relation_coefficients}
The coefficients of the polynomials \(q_{k,i}(x)\) and \(\tilde{q}_{k,i}(x)\) satisfy the relation \(\gamma_{n}^{k,i}=\frac{1}{n}\tilde{\gamma}_{n-1}^{k,i}\) for all \(k \in \mathbb{N}\), \(i \in \mathbb{Z}\) and \(n=1,\dots ,k\).
\end{prop}

\begin{proof}
By Lemma \ref{relation_coefficients2} the claim of this lemma is equivalent to
	\begin{align} \label{Equation1}
	\gamma_{n}^{k,i}=\frac{1}{n}\sum_{j=1}^{k-n+1}(-1)^{j}\alpha^{1}_{k,j}\gamma_{n-1}^{k-j,i} \quad \text{for}~ i \in \mathbb{Z}~ \text{and}~ n=1,\dots, k.
	\end{align}
First, let \(n=k\). Using Lemma \ref{recursion_gamma} and the fact that \(\alpha^{1}_{k,1}=2k\) we obtain
	\begin{align*}
	\gamma_{k}^{k,i}=-2\gamma_{k-1}^{k-1,i}=-\frac{1}{k}\alpha^{1}_{k,1}\gamma_{k-1}^{k-1,i}.
	\end{align*}
This proves (\ref{Equation1}) for \(n=k\) and all \(i \in \mathbb{Z}\). Now, let \(n \in \lbrace 1,\dots,k-1 \rbrace\). By Definition \ref{def_coefficients_c} we have
	\begin{align*}
	\alpha^{1}_{k,k-n+1}=(-1)^{k}2^{-n+1}\Big(n\gamma_{n}^{k}-\sum_{j=1}^{k-n}(-1)^j\alpha^{1}_{k,j}\gamma_{n-1}^{k-j}\Big).
	\end{align*}
Using the fact that \(\gamma_{n-1}^{n-1}=(-1)^{n-1}2^{n-1}\) leads to
	\begin{align*}
	n\gamma_{n}^{k}&=(-1)^{k}2^{n-1}\alpha^{1}_{k,k-n+1}+\sum_{j=1}^{k-n}(-1)^j\alpha^{1}_{k,j}\gamma_{n-1}^{k-j}\\
	&=(-1)^{k-n+1}\alpha^{1}_{k,k-n+1}\gamma_{n-1}^{n-1}+\sum_{j=1}^{k-n}(-1)^j\alpha^{1}_{k,j}\gamma_{n-1}^{k-j}\\
	&=\sum_{j=1}^{k-n+1}(-1)^j\alpha^{1}_{k,j}\gamma_{n-1}^{k-j}.
	\end{align*}
Next, we show that this implies that (\ref{Equation1}) is true for any \(i \in \mathbb{Z}\). Let \(n \in \lbrace 1,\dots,k\rbrace\). From above for any \(r \in \lbrace n,...,k \rbrace\) we just saw that
	\begin{align*}
	r\gamma_{r}^{k}&=\sum_{j=1}^{k-r+1}(-1)^j\alpha^{1}_{k,j}\gamma_{r-1}^{k-j}.
	\end{align*}
Since $r\binom{r-1}{n-1}=n\binom{r}{n}$, the latter implies that $n\binom{r}{n}\gamma_{r}^{k}=\displaystyle{\sum_{j=1}^{k-r+1}(-1)^j\alpha^{1}_{k,j}\binom{r-1}{n-1}\gamma_{r-1}^{k-j}}$.
Multiplying by the term \((-1)^{n+r}\Big(\frac{i}{2}\Big)^{r-n}\) on both sides and summing up for \(r\) from $n$ to \(k\) leads to
	\begin{align*}
	n\sum_{r=n}^{k}(-1)^{r+n}\gamma_{r}^{k}\binom{r}{n}\Big(\frac{i}{2}\Big)^{r-n}&=\sum_{r=n}^{k}(-1)^{r+n}\Big( \sum_{j=1}^{k-r+1}(-1)^j\alpha^{1}_{k,j}\binom{r-1}{n-1}\gamma_{r-1}^{k-j}\Big(\frac{i}{2}\Big)^{r-n}\Big)\\
	&=\sum_{r=n-1}^{k-1}\Big( \sum_{j=1}^{k-r}(-1)^{r+n-1}(-1)^j\alpha^{1}_{k,j}\binom{r}{n-1}\gamma_{r}^{k-j}\Big(\frac{i}{2}\Big)^{r-(n-1)}\Big)\\
	&=\sum_{j=1}^{k-n+1}(-1)^j\alpha^{1}_{k,j}\Big( \sum_{r=n-1}^{k-j}(-1)^{r+n-1}\binom{r}{n-1}\gamma_{r}^{k-j}\Big(\frac{i}{2}\Big)^{r-(n-1)}\Big).
	\end{align*}
By \eqref{eq:gammai-gamma}
this implies that $\displaystyle{\gamma_{n}^{k,i}=\sum_{j=1}^{k-n+1}(-1)^{j}\alpha^{1}_{k,j}\gamma_{n-1}^{k-j,i}},$
	which concludes the proof.
\end{proof}

\section{Algebraic conditions for polynomial reproduction of a Hermite scheme of order \(d=2\) }

In this section we give algebraic conditions on the mask coefficients of a Hermite subdivision scheme of order \(d=2\) which ensures polynomial reproduction up to degree \(m\). 
The main result is stated below. Its proof is split into several Lemmata.

\begin{thm} \label{main_result_new}
Let \(H_{\mathcal{A}}\) be a Hermite subdivision scheme with parametrization \(\tau\). Then \(H_{\mathcal{A}}\) reproduces constants if and only if
	\begin{align}
	\textbf{A}(-1)\textbf{e}_{1,2}&=\textbf{0}_{2}, \label{condition1_k=0_new} \\
\textbf{A}(1)\textbf{e}_{1,2}&=2\textbf{e}_{1,2}. \label{condition2_k=0_new}
	\end{align}
Moreover,  \(H_{\mathcal{A}}\) reproduces polynomials up to degree \(m\geqslant 1\) if and only if it reproduces constants and 
	\begin{align}
	\textbf{A}^{(k)}(-1) \textbf{e}_{1,2}+\sum_{\ell =1}^k \alpha^{1}_{k,\ell}\cdot \textbf{A}^{(k-\ell)}(-1) \textbf{e}_{2,2}&=\textbf{0}_{2}, \label{condition1_new}\\
\textbf{A}^{(k)}(1) \textbf{e}_{1,2}+\sum_{\ell =1}^k \tilde{\alpha}^{1}_{k,\ell} \cdot \textbf{A}^{(k-\ell)}(1)  \textbf{e}_{2,2}&= 
\begin{bmatrix} 2q_{k,2\tau}(-\frac{\tau}{2}) \\\tilde{q}_{k,2\tau}(-\frac{\tau}{2}) \end{bmatrix}\label{condition2_new}
	\end{align}
for all \(k=1,\dots ,m\) with \(\tilde{\alpha}^{1}_{k,\ell}=(-1)^{\ell}\alpha^{1}_{k,\ell},\ \ell=1,\dots,k\), and $\alpha^{1}_{k,\ell}$ as in Definition \ref{def_coefficients_c}.
\end{thm}

We prove the first part of Theorem \ref{main_result_new} directly by presenting it as a separated Lemma. First some  important observations are made.
\begin{Rem}  It is worthwhile to remark that:
\begin{itemize}
\item[1)] Up to the reproduction of linear polynomials the algebraic conditions given in the theorem above are also given in \cite{dyn2}, though presented in a different way;
\item[2)] When $m=1$ the previous result allows us to identify the correct parametrization  corresponding to the choice $\tau=(\textbf{A}^{(1)}(1))_{11}$;
\item[3)] The entries of the right-hand side (\ref{condition2_new}) are $$2q_{k,2\tau}(-\frac{\tau}{2})=2\displaystyle{\prod_{r=0}^{k-1}(\tau-r)},\quad \tilde{q}_{k,2\tau}(-\frac{\tau}{2})=\sum_{n=1}^k(-1)^n\alpha_{k,n}^1\prod_{r=0}^{k-n-1}(\tau-r).$$
\end{itemize}
\end{Rem}

\begin{lem} \label{lemma_reproduction_constants_new}
A Hermite subdivision scheme \(H_{\mathcal{A}}\) reproduces constants if and only if (\ref{condition1_k=0_new}) and (\ref{condition2_k=0_new}) are satisfied.
\end{lem}

\begin{proof}
The reproduction of constants is equivalent to
	\begin{align*}
	\sum_{j\in \mathbb{Z}}A_{i-2j}\textbf{e}_{1,2}=\textbf{e}_{1,2} \quad \forall~i \in \mathbb{Z}.
	\end{align*}
We observe that from \eqref{condition1_k=0_new} and \eqref{condition2_k=0_new} we trivially have
$$
	2\textbf{e}_{1,2}{=}(\textbf{A}(1)+\textbf{A}(-1)) \textbf{e}_{1,2}=2\sum_{i\in \mathbb{Z}}A_{2i}\textbf{e}_{1,2}\quad 
\hbox{and} \quad
	2\textbf{e}_{1,2}{=}(\textbf{A}(1)-\textbf{A}(-1)) \textbf{e}_{1,2}=2\sum_{i\in \mathbb{Z}}A_{2i+1}\textbf{e}_{1,2},
$$
	which is the claim.
\end{proof}

Note that the reproduction of constants does not depend on the chosen parametrization. This is not surprising, since it is so in the scalar situation as well, see \cite{hormann}.

\begin{lem} \label{lemma1_new}
Let \(m\geqslant 1\). Then, condition (\ref{condition1_new}) is satisfied if and only if 
	\begin{align*}
	\textbf{A}_e^{(k)}(1) \textbf{e}_{1,2}+\sum_{\ell =1}^k (-1)^{\ell}\alpha^{1}_{k,\ell} \textbf{A}_e^{(k-\ell)}(1)  \textbf{e}_{2,2}&=
\frac{1}{2}\Big(\textbf{A}^{(k)}(1) \textbf{e}_{1,2}+\sum_{\ell =1}^k (-1)^{\ell}\alpha^{1}_{k,\ell} \textbf{A}^{(k-\ell)}(1)  \textbf{e}_{2,2}\Big),\\
\textbf{A}_o^{(k)}(1) \textbf{e}_{1,2}+\sum_{\ell =1}^k (-1)^{\ell}\alpha^{1}_{k,\ell} \textbf{A}_o^{(k-\ell)}(1)  \textbf{e}_{2,2}&=
\frac{1}{2}\Big(\textbf{A}^{(k)}(1) \textbf{e}_{1,2}+\sum_{\ell =1}^k (-1)^{\ell}\alpha^{1}_{k,\ell} \textbf{A}^{(k-\ell)}(1)  \textbf{e}_{2,2}\Big),
	\end{align*}
for all \(k=1,\dots,m\). Moreover, condition (\ref{condition1_k=0_new}) is satisfied if and only if 
	\begin{align*}
	\textbf{A}_e(1)\textbf{e}_{1,2}=\textbf{A}_o(1)\textbf{e}_{1,2}=\frac{1}{2}\textbf{A}(1)\textbf{e}_{1,2}.
	\end{align*}
\end{lem}

\begin{proof}
Let \(k \in \lbrace 1,\dots,m\rbrace\).We have \(\textbf{A}^{(k)}(z)=\textbf{A}_e^{(k)}(z)+\textbf{A}_o^{(k)}(z)\) and therefore especially
	\begin{align} \label{equation_1_new}
	\textbf{A}^{(k)}(1)=\textbf{A}_e^{(k)}(1)+\textbf{A}_o^{(k)}(1)
	\end{align}
and $\textbf{A}^{(k)}(-1)=(-1)^k\textbf{A}_e^{(k)}(1)+(-1)^{k+1}\textbf{A}_o^{(k)}(1)$.
So, condition (\ref{condition1_new}) is equivalent to
	\begin{equation} \label{equation_2_new}
	\textbf{A}_e^{(k)}(1)\textbf{e}_{1,2}+\sum_{\ell =1}^k (-1)^{\ell}\alpha^{1}_{k,\ell} \textbf{A}_e^{(k-\ell)}(1)  \textbf{e}_{2,2}
	=\textbf{A}_o^{(k)}(1) \textbf{e}_{1,2}+\sum_{\ell =1}^k (-1)^{\ell}\alpha^{1}_{k,\ell} \textbf{A}_o^{(k-\ell)}(1)  \textbf{e}_{2,2}. 
	\end{equation}
Now, using  (\ref{equation_1_new}) we write 

	\begin{align*}
	&\textbf{A}^{(k)}(1) \textbf{e}_{1,2}+\sum_{\ell =1}^k (-1)^{\ell}\alpha^{1}_{k,\ell} \textbf{A}^{(k-\ell)}(1)  \textbf{e}_{2,2}\\
	&=\textbf{A}_e^{(k)}(1) \textbf{e}_{1,2}+\sum_{\ell =1}^k (-1)^{\ell}\alpha^{1}_{k,\ell} \textbf{A}_e^{(k-\ell)}(1)  \textbf{e}_{2,2}+\textbf{A}_o^{(k)}(1) \textbf{e}_{1,2}+\sum_{\ell =1}^k (-1)^{\ell}\alpha^{1}_{k,\ell} \textbf{A}_o^{(k-\ell)}(1)  \textbf{e}_{2,2}
	\end{align*}
which, together with (\ref{equation_2_new}), proves the first part of the Lemma.
Condition (\ref{condition1_k=0_new}) is equivalent to \(\textbf{A}_{e}(1)\textbf{e}_{1,2}=\textbf{A}_{o}(1)\textbf{e}_{1,2}\) and since $\textbf{A}(1)=\textbf{A}_{e}(1)+\textbf{A}_{o}(1)$,
 the claim is proved in this case as well.
\end{proof}

In the following we make use of the polynomials \(q_{k,i}\) and \(\tilde{q}_{k,i}\) introduced in the previous section. First, we unite them into the vector polynomial 
	\begin{align*}
	Q_{k,i}(x):= \begin{bmatrix} q_{k,i}(x) \\ \tilde{q}_{k,i}(x)\end{bmatrix} \quad \text{with} \quad k\geqslant 0,~i\in \mathbb{Z}.
	\end{align*}

\begin{prop} \label{corollary1_new}
Let \(m\geqslant 1\). Then, conditions (\ref{condition1_new}) and (\ref{condition2_new}) are satisfied if and only if for all $i\in\mathbb{Z}$ and \(\tau \in \mathbb{R}\), 
	\begin{equation}\label{eq:Costi}
	\sum_{j\in \mathbb{Z}}A_{i-2j}Q_{k,i+2\tau}(-j-\tau)=\begin{bmatrix} q_{k,i+2\tau}(\frac{-i-\tau}{2}) \\ \frac{1}{2}\tilde{q}_{k,i+2\tau}(\frac{-i-\tau}{2})\end{bmatrix} \quad ~ k=1,\dots,m.	\end{equation}
Especially, conditions (\ref{condition1_k=0_new}) and (\ref{condition2_k=0_new}) are satisfied if and only if  \(\sum_{j\in \mathbb{Z}}A_{i-2j}Q_{0,i}(-j)=\textbf{e}_{1,2}\) for all \(i \in \mathbb{Z}\).
\end{prop}

\begin{proof}
Observe that by definition of the class of polynomials in (\ref{def_polynomial_1}) we obtain
	\begin{align*}
	\textbf{A}_e^{(k)}(1)&=\sum_{j\in \mathbb{Z}}q_{k,2(j-t)}(t)A_{2j}=\sum_{j\in \mathbb{Z}}q_{k,2t+2\tau}(-j-\tau)A_{2(t-j)},\\
		\textbf{A}_o^{(k)}(1)&=\sum_{j\in \mathbb{Z}}q_{k,2(j-t)+1}(t)A_{2j+1}=\sum_{j\in \mathbb{Z}}q_{k,2t+2\tau+1}(-j-\tau)A_{2(t-j)+1}.
	\end{align*}
for all \(t\in \mathbb{Z}\) and $\tau\in \mathbb{R}$. Let \(i \in 2\mathbb{Z}\) with \(i=2s\) for some \(s \in \mathbb{Z}\). 
This observation together with Lemma \ref{lemma1_new} implies 
	\begin{align*}
	\sum_{j\in \mathbb{Z}}&A_{i-2j}Q_{k,i+2\tau}(-j-\tau)=\sum_{j\in \mathbb{Z}}A_{2(s-j)}Q_{k,2s+2\tau}(-j-\tau) \\
	&=\sum_{j\in \mathbb{Z}}q_{k,2s+2\tau}(-j-\tau)A_{2(s-j)}\textbf{e}_{1,2}+\sum_{j\in \mathbb{Z}}\tilde{q}_{k,2s+2\tau}(-j-\tau)A_{2(s-j)}\textbf{e}_{2,2} \\ 
	&= \sum_{j\in \mathbb{Z}}q_{k,2s+2\tau}(-j-\tau)A_{2(s-j)}\textbf{e}_{1,2}+\sum_{j\in \mathbb{Z}}\sum_{\ell=1}^{k}(-1)^{\ell}\alpha^{1}_{k,\ell}q_{k-\ell,2s+2\tau}(-j-\tau)A_{2(s-j)}\textbf{e}_{2,2} \\
	&=\textbf{A}_{e}^{(k)}(1) \textbf{e}_{1,2}+\sum_{\ell =1}^k (-1)^{\ell}\alpha^{1}_{k,\ell}\textbf{A}_{e}^{(k-\ell)}(1)\textbf{e}_{2,2}\\
	&=\frac{1}{2}\Big(\textbf{A}^{(k)}(1) \textbf{e}_{1,2}+\sum_{\ell =1}^k (-1)^{\ell}\alpha^{1}_{k,\ell}\textbf{A}^{(k-\ell)}(1) \textbf{e}_{2,2} \Big)\\
	&=\begin{bmatrix} q_{k,i+2\tau}(\frac{-i-\tau}{2}) \\ \frac{1}{2}\tilde{q}_{k,i+2\tau}(\frac{-i-\tau}{2})\end{bmatrix},\qquad \hbox{showing the claim for $i$ even.}
	\end{align*}

Similarly, for odd \(i\in \mathbb{Z}\), \(i=2s+1\), we obtain that $\displaystyle{\sum_{j\in \mathbb{Z}}A_{i-2j}Q_{k,i+2\tau}(-j-\tau)=\begin{bmatrix} q_{k,i+2\tau}(\frac{-i-\tau}{2}) \\ \frac{1}{2}\tilde{q}_{k,i+2\tau}(\frac{-i-\tau}{2})\end{bmatrix}}.$

The second part of the corollary follows by Lemma \ref{lemma_reproduction_constants_new} and \(q_{0,i}=1\) and \(\tilde{q}_{0,i}=0\).
\end{proof}

Note that the right-hand side \eqref{eq:Costi} actually does not depend on \(i \in \mathbb{Z}\) since $q_{k,i+2\tau}(\frac{-i-\tau}{2})=q_{k,2\tau}(\frac{-\tau}{2}) $.

\begin{prop} \label{prop_reproduction_new}
Let \(H_{\mathcal{A}}\) be a Hermite subdivision scheme with parametrization \(\tau\) and \(m\geqslant 0\). Then, \(H_{\mathcal{A}}\) satisfies conditions \eqref{condition1_k=0_new} -- \eqref{condition2_k=0_new} and
conditions (\ref{condition1_new}) -- (\ref{condition2_new})  for all \(k=1,\dots,m\),  if and only if				
	\begin{align}
	\sum_{j\in \mathbb{Z}} A_{i-2j}\textbf{e}_{1,2}&=\textbf{e}_{1,2} \quad i\in\mathbb{Z}, \notag\\
	\sum_{j\in \mathbb{Z}} A_{i-2j}\begin{bmatrix} (j+\tau)^{k} \\ k(j+\tau)^{k-1}\end{bmatrix}&=\frac{1}{2^k}\begin{bmatrix} (i+\tau)^k \\ k(i+\tau)^{k-1}\end{bmatrix} \quad k=1,\dots,m, \quad i\in \mathbb{Z}\label{second2} ,
	\end{align}
	with the convention that \eqref{second2} (resp. (\ref{condition1_new}) -- (\ref{condition2_new})) is empty if $m=0$.
\end{prop}

\begin{proof}
We prove the proposition by induction over \(m\). The case \(m=0\) follows by Lemma \ref{lemma_reproduction_constants_new}. Assume that the statement is true for some \(m-1\) and all \(k=1,\dots ,m-1\). 
The proof use the representations of the polynomials \(q_{k,\ell}(x)\) and \(\tilde{q}_{k,\ell}(x)\) as in (\ref{coefficients1}) and (\ref{coefficients2}). For \(i\in\mathbb{Z}\),  using Proposition 
\ref{corollary1_new} we obtain
	\begin{align*}
&\begin{bmatrix} q_{m,i+2\tau}(\frac{-i-\tau}{2}) \\ \frac{1}{2}\tilde{q}_{m,i+2\tau}(\frac{-i-\tau}{2})\end{bmatrix}=\sum_{j\in \mathbb{Z}}A_{i-2j}Q_{m,i+2\tau}(-j-\tau)\\
	&=\sum_{j\in \mathbb{Z}}A_{i-2j}q_{m,i+2\tau}(-j-\tau)\textbf{e}_{1,2}+\sum_{j\in \mathbb{Z}}A_{i-2j}\tilde{q}_{m,i+2\tau}(-j-\tau)\textbf{e}_{2,2}\\
	&=\sum_{j\in \mathbb{Z}}A_{i-2j} \sum_{n=0}^{m}\gamma_{n}^{m,i+2\tau}\begin{bmatrix} (j+\tau)^n \\ 0\end{bmatrix}+\sum_{j\in \mathbb{Z}}A_{i-2j}\sum_{n=0}^{m-1}\tilde{\gamma}^{m,i+2\tau}_{n}\begin{bmatrix} 0 \\ (j+\tau)^n\end{bmatrix}\\
	&=\gamma_{m}^{m,i+2\tau}\sum_{j\in \mathbb{Z}}A_{i-2j}\begin{bmatrix} (j+\tau)^m \\ m(j+\tau)^{m-1}\end{bmatrix}+\\
	&\underbrace{\sum_{j\in \mathbb{Z}}A_{i-2j} \sum_{n=0}^{m-1}\gamma_{n}^{m,i+2\tau}\begin{bmatrix} (j+\tau)^n \\ 0\end{bmatrix}+\sum_{j\in \mathbb{Z}}A_{i-2j}\sum_{n=0}^{m-2}\tilde{\gamma}^{m,i+2\tau}_{n}\begin{bmatrix} 0 \\ (j+\tau)^n\end{bmatrix}}_{\text{($\ast$)}}.
	\end{align*}
Note that we used the relation \(\tilde{\gamma}_{m-1}^{m,i+2\tau}=m\gamma_{m}^{m,i+2\tau}\) to obtain the last equality above, see Proposition \ref{relation_coefficients}. Before we can apply the induction hypothesis to ($\ast$) we apply Proposition  \ref{relation_coefficients} again and get
	\begin{align*}
	&\sum_{j\in \mathbb{Z}}A_{i-2j} \sum_{n=0}^{m-1}\gamma_{n}^{m,i+2\tau}\begin{bmatrix} (j+\tau)^n \\ 0\end{bmatrix}+\sum_{j\in \mathbb{Z}}A_{i-2j}\sum_{n=0}^{m-2}\tilde{\gamma}^{m,i+2\tau}_{n}\begin{bmatrix} 0 \\ (j+\tau)^n\end{bmatrix} \\
	&=\gamma_{0}^{m,i+2\tau}\sum_{j\in \mathbb{Z}}A_{i-2j}\textbf{e}_{1,2}+ \sum_{n=0}^{m-2}\Big(\gamma_{n+1}^{m,i+2\tau}\sum_{j\in \mathbb{Z}}A_{i-2j}\begin{bmatrix} (j+\tau)^{n+1} \\ 0\end{bmatrix}+\tilde{\gamma}_{n}^{m,i+2\tau}\sum_{j\in \mathbb{Z}}A_{i-2j}\begin{bmatrix} 0 \\ (j+\tau)^{n}\end{bmatrix}\Big)\\
	&=\gamma_{0}^{m,i+2\tau}\sum_{j\in \mathbb{Z}}A_{i-2j}\textbf{e}_{1,2}+ \sum_{n=0}^{m-2}\gamma_{n+1}^{m,i+2\tau}\sum_{j\in \mathbb{Z}}A_{i-2j}\begin{bmatrix} (j+\tau)^{n+1} \\ (n+1)(j+\tau)^{n}\end{bmatrix}. \\
	\end{align*}
Now we use the assumption that the scheme reproduces constants for the first part of the right-hand side and apply the induction hypothesis for \(k=1,\dots,m-1\) to the second part. Therefore,  	
	\begin{align*}
	\sum_{j\in \mathbb{Z}}A_{i-2j} \sum_{n=0}^{m-1}\gamma_{n}^{m,i+2\tau}\begin{bmatrix} (j+\tau)^n \\ 0\end{bmatrix}&+\sum_{j\in \mathbb{Z}}A_{i-2j}\sum_{n=0}^{m-2}\tilde{\gamma}^{m,i+2\tau}_{n}\begin{bmatrix} 0 \\ (j+\tau)^n\end{bmatrix} \\
	&=\underbrace{\gamma_{0}^{m,i+2\tau}\textbf{e}_{1,2}+ \sum_{n=0}^{m-2}\frac{1}{2^{n+1}}\gamma_{n+1}^{m,i+2\tau} \begin{bmatrix} (i+\tau)^{n+1} \\ (n+1)(i+\tau)^{n}\end{bmatrix}}_{\text{($\ast \ast$)}}.
	\end{align*}
The next step is to rewrite the sum ($\ast \ast$) by first applying Proposition \ref{relation_coefficients} and then using the definition of the polynomial \(q_{m,i}\) (resp. \(\tilde{q}_{m,i}\)) as in (\ref{coefficients1}) (resp.(\ref{coefficients2})). So, 
	\begin{align*}
	\gamma_{0}^{m,i+2\tau}&\textbf{e}_{1,2}+ \sum_{n=0}^{m-2}\frac{1}{2^{n+1}}\gamma_{n+1}^{m,i+2\tau} \begin{bmatrix} (i+\tau)^{n+1} \\ (n+1)(i+\tau)^{n}\end{bmatrix}\\
	&=\gamma_{0}^{m,i+2\tau}\textbf{e}_{1,2}+ \sum_{n=0}^{m-1}\gamma_{n+1}^{m,i+2\tau} \Big(\frac{i+\tau}{2}\Big)^{n+1}\textbf{e}_{1,2}\\
	&~~+\frac{1}{2}\sum_{n=0}^{m-1}\tilde{\gamma}_{n}^{m,i+2\tau} \Big(\frac{i+\tau}{2}\Big)^{n}\textbf{e}_{2,2}-\gamma_{m}^{m,i+2\tau}\Big(\frac{i+\tau}{2}\Big)^{m}\textbf{e}_{1,2}-\frac{1}{2}\tilde{\gamma}^{m,i+2\tau}_{m-1}\Big(\frac{i+\tau}{2}\Big)^{m-1}\textbf{e}_{2,2}\\
	&=\gamma_{0}^{m,i+2\tau}\textbf{e}_{1,2}+q_{m,i+2\tau}\Big(-\frac{i+\tau}{2}\Big)\textbf{e}_{1,2}-\gamma_{0}^{m,i+2\tau}\textbf{e}_{1,2}+ \frac{1}{2} \tilde{q}_{m,i+2\tau}\Big(-\frac{i+\tau}{2}\Big)\textbf{e}_{2,2}\\
		&~~-\gamma_{m}^{m,i+2\tau}\Big(\frac{i+\tau}{2}\Big)^{m}\textbf{e}_{1,2}-\frac{1}{2}\tilde{\gamma}^{m,i+2\tau}_{m-1}\Big(\frac{i+\tau}{2}\Big)^{m-1}\textbf{e}_{2,2}\\
		&=q_{m,i+2\tau}\Big(-\frac{i+\tau}{2}\Big)\textbf{e}_{1,2}+ \frac{1}{2} \tilde{q}_{m,i+2\tau}\Big(-\frac{i+\tau}{2}\Big)\textbf{e}_{2,2}-\gamma_{m}^{m,i+2\tau}\Big(\frac{i+\tau}{2}\Big)^{m}\textbf{e}_{1,2}\\
		&~~-\frac{1}{2}\tilde{\gamma}^{m,i+2\tau}_{m-1}\Big(\frac{i+\tau}{2}\Big)^{m-1}\textbf{e}_{2,2}.
	\end{align*}	
We apply Proposition \ref{relation_coefficients} to the right-hand side above to obtain 
$$
\gamma_{0}^{m,i+2\tau}\textbf{e}_{1,2}+ \frac{1}{2^{n+1}}\sum_{n=0}^{m-2}\gamma_{n+1}^{m,i+2\tau} \begin{bmatrix} (i+\tau)^{n+1} \\ (n+1)(i+\tau)^{n}\end{bmatrix}=\begin{bmatrix} q_{m,i+2\tau}(\frac{-i-\tau}{2}) \\ \frac{1}{2}\tilde{q}_{m,i+2\tau}(\frac{-i-\tau}{2})\end{bmatrix}-\gamma_{m}^{m,i+2\tau}\Big(\frac{1}{2}\Big)^{m}\begin{bmatrix} (i+\tau)^m \\ m(i+\tau)^{m-1}\end{bmatrix}.
$$
Summarizing our previous computations leads to 
	\begin{align*}
	\gamma_{m}^{m,i+2\tau}\sum_{j\in \mathbb{Z}}A_{i-2j}\begin{bmatrix} (j+\tau)^m \\ m(j+\tau)^{m-1}\end{bmatrix}&=\gamma_{m}^{m,i+2\tau}\Big(\frac{1}{2}\Big)^{m}\begin{bmatrix} (i+\tau)^m \\ m(i+\tau)^{m-1}\end{bmatrix}.
	\end{align*}
Since \(\gamma_{m}^{m,i+2\tau}\neq 0\), this 
	concludes the induction step.
\end{proof}
\begin{Rem} If \(m=1\) the sums in the proof above which are not defined are assumed to be zero. The conclusion of the Proposition in this case is still true.
\end{Rem}

\smallskip We are finally in a position to prove Theorem  \ref{main_result_new}.
\begin{proof}[Proof of Theorem \ref{main_result_new}]
We prove the statement by induction over the degree \(m\) of the polynomials. For \(m=0\) we refer to Lemma \ref{lemma_reproduction_constants_new}. 
So, for $m\geqslant 1$ assume the statement is true for some \(m-1\) and all \(k=0,\dots ,m-1\) and show that the Hermite subdivision scheme \(H_{\mathcal{A}}\) reproduces polynomials of degree \(m\). Let \(p(x)=x^m+g(x)\) with \(g(x) \in \prod_{m-1}\). By Definition \ref{def_reproduction} we have to show that for $n\in \mathbb{N}$ and $i\in \mathbb{Z}$,
	\begin{align*}
	\textbf{f}_{n}(i)=\begin{bmatrix} p((i+\tau)/2^n) \\ p^\prime((i+\tau)/2^n)\end{bmatrix} \quad  \Longrightarrow \quad \textbf{f}_{n+1}(i)=\begin{bmatrix} p((i+\tau)/2^{n+1}) \\ p^\prime((i+\tau)/2^{n+1})\end{bmatrix}.
	\end{align*}
Let \(n \in \mathbb{N}\) and \(i \in \mathbb{Z}\). 
By (\ref{subdivision_rule}) we have
	\begin{align*}
	\textbf{D}^{n+1}\textbf{f}_{n+1}(i)=\sum_{j \in \mathbb{Z}}A_{i-2j}\textbf{D}^n\begin{bmatrix} ((j+\tau)/2^n)^m \\m((j+\tau)/2^n)^{m-1}\end{bmatrix}+\sum_{j \in \mathbb{Z}}A_{i-2j}\textbf{D}^n\begin{bmatrix} g((j+\tau)/2^{n}) \\ g^\prime((j+\tau)/2^{n})\end{bmatrix}. 
	\end{align*}
This is equivalent to 
	\begin{align*}
	2^{nm}\textbf{D}^{n+1}\textbf{f}_{n+1}(i)
	=\sum_{j \in \mathbb{Z}}A_{i-2j}\begin{bmatrix} (j+\tau)^m \\m(j+\tau)^{m-1}\end{bmatrix}+2^{nm}\sum_{j \in \mathbb{Z}}A_{i-2j}\textbf{D}^n\begin{bmatrix} g((j+\tau)/2^{n}) \\ g^\prime((j+\tau)/2^{n})\end{bmatrix}.
	\end{align*}
Now, we apply Proposition \ref{prop_reproduction_new} to the first summand of the right-hand side above and the induction hypothesis to the second. We obtain 
	\begin{align*}
	2^{nm}\textbf{D}^{n+1}\textbf{f}_{n+1}(i)=\frac{1}{2^m}
	\begin{bmatrix} (i+\tau)^m \\m(i+\tau)^{m-1}\end{bmatrix}+2^{nm}\textbf{D}^{n+1}\begin{bmatrix} g((i+\tau)/2^{n+1}) \\ g^\prime((i+\tau)/2^{n+1})\end{bmatrix}.
	\end{align*}
So, we see that $\textbf{D}^{n+1}\textbf{f}_{n+1}(i)=\begin{bmatrix} 2^{-(n+1)m}(i+\tau)^m +g((i+\tau)/2^{n+1}) \\2^{-(n+1)m}m(i+\tau)^{m-1}+2^{-(n+1)}g^\prime((i+\tau)/2^{n+1}) \end{bmatrix}$.
This is equivalent to $\textbf{f}_{n+1}(i)=\begin{bmatrix} ((i+\tau)/2^{n+1})^m +g((i+\tau)/2^{n+1}) \\m((i+\tau)/2^{n+1})^{m-1}+g^\prime((i+\tau)/2^{n+1}) \end{bmatrix}
$
which proves the claim. 
\end{proof}

\section{Example of Hermite schemes of order $d=2$} \label{ex:interp}
We start with the interpolatory Hermite subdivision scheme \(H_{\mathcal{A}}\) introduced by Merrien in \cite{merrien1}. 
Its non-zero mask coefficients are given by
	\begin{align*}
	A_{-1}=\begin{bmatrix} \frac{1}{2} & \lambda\\ \frac{1}{2}(1-\mu) & \frac{\mu}{4}\end{bmatrix}, \quad \quad 
	A_{0}=\begin{bmatrix} 1 & 0\\ 0& \frac{1}{2}\end{bmatrix}, \quad \quad
	A_{1}=\begin{bmatrix} \frac{1}{2} & -\lambda\\ \frac{1}{2}(\mu-1)& \frac{\mu}{4}\end{bmatrix}. 
\	
	\end{align*}
It is known that the scheme reproduces polynomials of degree 1 for all \(\lambda\), \(\mu \in \mathbb{R}\). It reproduces \(\prod_2\) if and only if \(\lambda=-\frac{1}{8}\). It reproduces polynomials of degree 3, if additionally \(\mu=-\frac{1}{2}\). Computations show that this result coincides with our conditions of Theorem \ref{main_result_new}. 
Moreover, it tells us that the Hermite scheme \(H_{\mathcal{A}}\) does not reproduces polynomials of degree 4 since
$$
	\textbf{A}^{(4)}(-1)\textbf{e}_{1,2}+8\textbf{A}^{(3)}(-1)\textbf{e}_{2,2}+12\textbf{A}^{(2)}(-1)\textbf{e}_{2,2} +16\textbf{A}^{(1)}(-1)\textbf{e}_{2,2}+12\textbf{A}(-1)\textbf{e}_{2,2}=\textbf{e}_{1,2}\neq \textbf{0}_{2}.
$$
Here, we use the values of the coefficients \(\alpha^{1}_{k,\ell}\) as given in Table \ref{table_c_kl}.

\smallskip Next, we use our algebraic conditions to obtain a new Hermite subdivision scheme which reproduces polynomials of degree higher than 3 by only slightly increasing the support of the scheme. Consider the interpolatory Hermite subdivision scheme \(H_{\mathcal{\bar{A}}}\) with non-zero mask coefficients
\begin{align*} 
	\bar{A}_{-3}=\begin{bmatrix} b_{1} & b_{2}\\ b_{3} & b_{4}\end{bmatrix}, \quad \bar{A}_{-1}=\begin{bmatrix} a_{1} & a_{2}\\ a_{3} & a_{4}\end{bmatrix}, \quad \bar{A}_{0}=\begin{bmatrix} 1 & 0\\ 0& \frac{1}{2}\end{bmatrix}, \quad \bar{A}_{1}=\begin{bmatrix} a_{1} & -a_{2}\\ -a_{3}& a_{4}\end{bmatrix}, \quad \bar{A}_{3}=\begin{bmatrix} b_{1} & -b_{2}\\ -b_{3}& b_{4}\end{bmatrix} 
\end{align*} 
	for some real coefficients \(a_{i},b_{i}\), \(i=1,\dots,4\).
By Theorem \ref{main_result_new} these coefficients have to satisfy the following linear system in order to reproduce polynomials up to degree 5
	\begin{align*}
	b_{1}&=\frac{1}{128}-3b_{2}, & &b_{4}=\frac{1}{1408}-\frac{384}{1408}b_{3}, & &a_{1}=\frac{1}{2}-b_{1}, \\ 
	a_{3}&=24b_{4}+9b_{3}+\frac{3}{4}, & &a_{4}=\frac{1}{4}-b_{4}-\frac{1}{2}a_{3}-\frac{3}{2}b_{3}, & &a_{2}=-\frac{1}{8}-3b_{2}-2b_{1}.
	\end{align*}
Choosing the coefficients \(b_{3}=0\) and \(b_{2}=\frac{1}{384}\) leads to \(a_{1}=\frac{1}{2}\), \(a_{2}= -17/128 \approx -0.13\), \(a_{3}=135/176 \approx 0.77\) and \(a_{4}=-189/1408\approx -0.13\). With this choice of coefficients the non-zero mask coefficients \(\bar{A}_{-1}, \bar{A}_{0}\) and \(\bar{A}_{1}\) of the scheme \(H_{\mathcal{\bar{A}}}\) are closely related to the corresponding ones of \(H_{\mathcal{A}}\). See Figure \ref{basic_limit} for the basic limit functions of the scheme \(H_{\mathcal{\bar{A}}}\).
\begin{figure}
\centering
\includegraphics[scale=1]{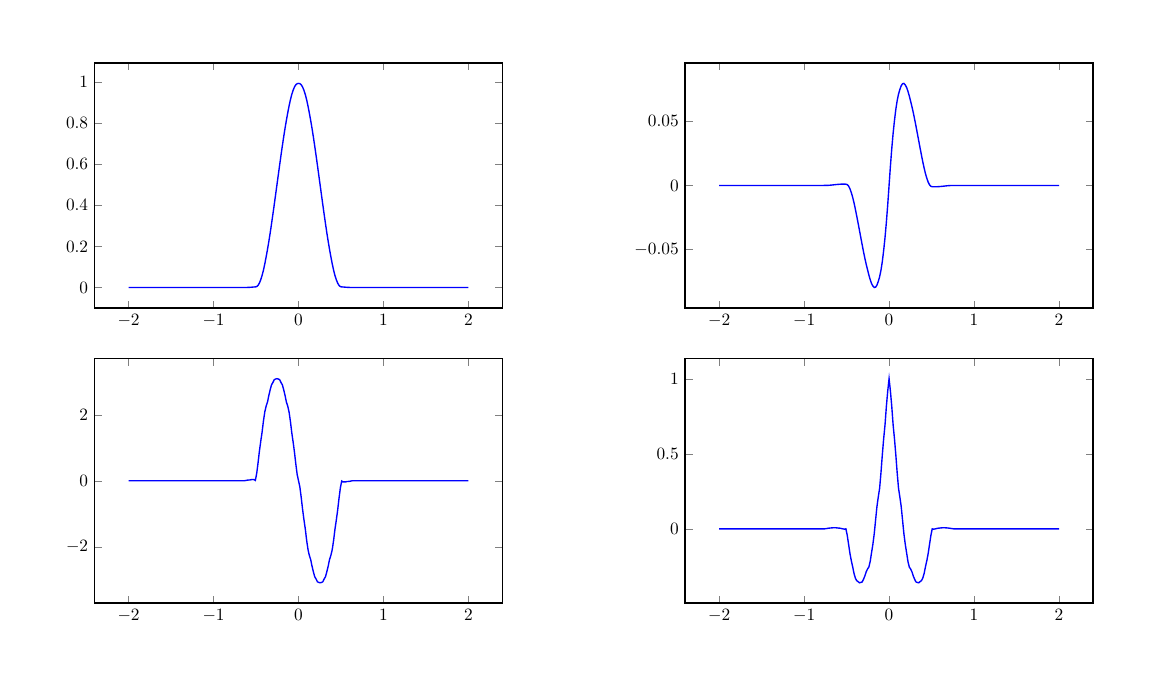}
\caption{Left column: Basic limit function and derivative of the interpolatory Hermite scheme \(H_{\mathcal{\bar{A}}}\) introduced in Section \ref{ex:interp} for initial data  \(\textbf{e}_{1,2}\) at 0 and \(\textbf{0}_{2}\) else. Right column: Basic limit function and derivative of the interpolatory Hermite scheme \(H_{\mathcal{\bar{A}}}\) introduced in Section \ref{ex:interp} for initial data  \(\textbf{e}_{2,2}\) at 0 and \(\textbf{0}_{2}\) else.} 
		\label{basic_limit}
\end{figure}

\medskip We now consider the de Rham transform of the interpolatory Hermite scheme \(H_{\mathcal{A}}\) as introduced in \cite{conti1}. This scheme is a dual scheme, meaning that \(\tau=-\frac{1}{2}\). For \(\lambda,\mu \in \mathbb{R}\) the non-zero matrices of its mask (for simplicity again denoted by \(A_{i}\)) are given by
	\begin{align*}
	A_{-2}=\frac{1}{8}\begin{bmatrix} 2+4\lambda(1-\mu)& 4\lambda+2\lambda\mu\\ 4-2\mu-2\mu^2& \mu^2+8\lambda(1-\mu)\end{bmatrix}, &
	\quad A_{-1}=\frac{1}{8}\begin{bmatrix} 6-4\lambda(1-\mu) & 8\lambda-2\lambda\mu\\ 4-2\mu-2\mu^2& \mu^2-8\lambda(1-\mu)+2\mu \end{bmatrix}, \\ 	
	A_{0}=\frac{1}{8}\begin{bmatrix} 6-4\lambda(1-\mu) & -8\lambda+2\lambda\mu\\ -4+2\mu+2\mu^2& \mu^2-8\lambda(1-\mu)+2\mu \end{bmatrix}, &
	\quad A_{1}=\frac{1}{8}\begin{bmatrix} 2+4\lambda(1-\mu)& -4\lambda-2\lambda\mu\\ -4+2\mu+2\mu^2& \mu^2+8\lambda(1-\mu)\end{bmatrix}.
	\end{align*}
We see that the scheme reproduces constants since it satisfies (\ref{condition1_k=0_new}) and (\ref{condition2_k=0_new}).
Furthermore, we obtain
	\begin{align*}
	\textbf{A}^{(1)}(-1)\textbf{e}_{1,2}+2\textbf{A}(-1)\textbf{e}_{2,2}&=\frac{1}{8}\begin{bmatrix} 16\lambda(1-\mu) \\ 0 \end{bmatrix}+\frac{2}{8}\begin{bmatrix} -8\lambda+8\lambda\mu \\ 0 \end{bmatrix}=\textbf{0}_{2},\\
	\textbf{A}^{(1)}(1)\textbf{e}_{1,2}-2\textbf{A}(1)\textbf{e}_{2,2}&=\frac{1}{8}\begin{bmatrix} -8 \\ -16+8\mu+8\mu^2\end{bmatrix}-\frac{2}{8}\begin{bmatrix} 0 \\ 4\mu^2+4\mu\end{bmatrix}=\begin{bmatrix} -1 \\ -2\end{bmatrix}.
	\end{align*}	 
Since \(2q_{1,2\tau}(-\frac{\tau}{2})=-1\) and \(\tilde{q}_{1,2\tau}(-\frac{\tau}{2})=-2\) we conclude that the scheme reproduces linear polynomials by Theorem \ref{main_result_new}. Next, we check if the scheme also reproduces polynomials of degree 2. By Table \ref{table_c_kl} we have \(\alpha^{1}_{2,1}=4\) and \(\alpha^{1}_{2,2}=2\). We compute 
\(2q_{2,2\tau}\Big(-\frac{\tau}{2}\Big)=\frac{3}{2}\) and \(\tilde{q}_{2,2\tau}\Big(-\frac{\tau}{2}\Big)=4.
\)
According to Theorem \ref{main_result_new} we have to check if the scheme satisfies the two relations
	\begin{align*}
	\textbf{A}^{(2)}(-1)\textbf{e}_{1,2}+4\textbf{A}^{(1)}(-1)\textbf{e}_{2,2}+2\textbf{A}(-1)\textbf{e}_{2,2}&=\textbf{0}_{2},\\
	\textbf{A}^{(2)}(1)\textbf{e}_{1,2}-4\textbf{A}^{(1)}(1)\textbf{e}_{2,2}+2\textbf{A}(1)\textbf{e}_{2,2}&=\begin{bmatrix} \frac{3}{2} \\ 4\end{bmatrix}
	\end{align*}
in order to decide whether it reproduces \(\prod_2\) or not.
Computations lead to
	\begin{align*}
	\textbf{A}^{(2)}(-1)\textbf{e}_{1,2}+4\textbf{A}^{(1)}(-1)\textbf{e}_{2,2}+2\textbf{A}(-1)\textbf{e}_{2,2}&=\begin{bmatrix} 0 \\ 2-2\mu+16\lambda-16\lambda\mu \end{bmatrix},\\
	\textbf{A}^{(2)}(1)\textbf{e}_{1,2}-4\textbf{A}^{(1)}(1)\textbf{e}_{2,2}+2\textbf{A}(1)\textbf{e}_{2,2}&=\begin{bmatrix} 3+12\lambda \\ 4\end{bmatrix}.
	\end{align*}
We conclude that the scheme reproduces polynomials up to degree 2 if and only if \(\lambda=-\frac{1}{8}\). Similar computations for the case of cubic polynomials show that the choice of \(\mu=-\frac{1}{2}\) leads to the reproduction of cubic polynomials. 


\section{Generalization to Hermite schemes of order $d=3$}
As a matter of fact, the known examples of Hermite schemes are of dimension $d=2,3$. Therefore, in this section, we generalize our previous results to Hermite subdivision schemes of order $d=3$. Nevertheless, since the proofs turn out to be similar, we here put in evidence only the differences. The first one is the need of a third auxiliary class of polynomials. 

\subsection{Polynomials $\hat{q}_k$}
Before we state the algebraic conditions for the reproduction property of Hermite subdivision schemes of order $d=3$, we introduce a new class of polynomials \(\hat{q}_{k,i}\) similar to the polynomials \(\tilde{q}_{k,i}\) in (\ref{def_polynomial_2}).
First, we recursively define a new set of coefficients related to the coefficients \(\alpha^1_{k,\ell},\ \ell=1,\dots,k\) of Definition \ref{def_coefficients_c}. 
\begin{Def}\label{coefficients2}
Let $k\in \mathbb{N}$, $k\geqslant 2$. We define
	\begin{align*}
	\begin{array}{ll}
	\alpha^{2}_{k,2}&:=4k(k-1) \\
	\alpha^{2}_{k,k-n+2}&:=(-1)^{k}2^{-n+2}\Big(n(n-1)\gamma_{n}^{k}-\sum_{j=2}^{k-n+1}(-1)^j\alpha^{2}_{k,j}\gamma_{n-2}^{k-j}\Big) \quad n=k-1,\dots ,2.
	\end{array}
	\end{align*}
\end{Def}
We continue by defining the new polynomials $\hat{q}_{k,i}$, $k\geqslant 0$, $i\in  \mathbb{Z}$, as
\begin{align} \label{def_polynomial_3}
\hat{q}_{0,i}:=0,\quad \hat{q}_{1,i}:=0,\quad \hat{q}_{k,i}(x):=\sum_{n=2}^{k}(-1)^{n}\alpha^{2}_{k,n}q_{k-n,i}(x)\quad k\geqslant 2,\  i \in \mathbb{Z}.
	\end{align}
For a fixed  \(k \geqslant 2\) the polynomial \(\hat{q}_{k,i}(x)\) is of degree \(k-2\). Thus, we can write it in the form
	\begin{align}
	\hat{q}_{k,i}(-x)&=\sum_{n=0}^{k-2}\hat{\gamma}_{n}^{k,i}x^n, \quad \hbox{for some coefficients} \quad \hat{\gamma}_{n}^{k,i}\in \mathbb{R}.
	 \label{coefficients3}
	\end{align}
Similarly as in Proposition \ref{relation_coefficients} one can show that the following relation between the coefficients of the polynomials in (\ref{coefficients1}) and (\ref{coefficients3}) holds true.
For all \(i \in \mathbb{Z}\) and \(k\geqslant 2\) we obtain
	\begin{align*}
	\gamma_{n}^{k,i}=\frac{1}{n(n-1)}\hat{\gamma}_{n-2}^{k,i} \quad \text{for}~~n=2,\dots, k.
	\end{align*}
We state the main theorem of the previous section for Hermite schemes of order $d=3$ whose proof is omitted since it follows the same lines of reasoning of Theorem \ref{main_result_new}.
\begin{thm}  \label{main_result_3}
Let \(H_{\mathcal{A}}\) be a Hermite subdivision scheme of order 3 with parametrization \(\tau\). Then \(H_{\mathcal{A}}\) reproduces constants if and only if
$$
	\textbf{A}(-1)\textbf{e}_{1,3}=\textbf{0}_{3},  \quad
\textbf{A}(1)\textbf{e}_{1,3}=2\textbf{e}_{1,3}. 
$$
Moreover, \(H_{\mathcal{A}}\) reproduces polynomials up to degree \(m \geqslant 1\) if and only if it reproduces constants and 
	\begin{align*}
	\textbf{A}^{(k)}(-1) \textbf{e}_{1,3}+\sum_{\ell =1}^k \alpha^{1}_{k,\ell}\cdot \textbf{A}^{(k-\ell)}(-1) \textbf{e}_{2,3}+\sum_{\ell =2}^k \alpha^{2}_{k,\ell}\cdot \textbf{A}^{(k-\ell)}(-1) \textbf{e}_{3,3}&=\textbf{0}_{3},\\
\textbf{A}^{(k)}(1) \textbf{e}_{1,3}+\sum_{\ell =1}^k \tilde{\alpha}^{1}_{k,\ell} \cdot \textbf{A}^{(k-\ell)}(1) \textbf{e}_{2,3}+\sum_{\ell =2}^k \tilde{\alpha}^{2}_{k,\ell} \cdot \textbf{A}^{(k-\ell)}(1)  \textbf{e}_{3,3}&= 
\begin{bmatrix} 2q_{k,2\tau}(\frac{-\tau}{2}) \\\tilde{q}_{k,2\tau}(\frac{-\tau}{2}) \\ \frac{1}{2}\hat{q}_{k,2\tau}(\frac{-\tau}{2})\end{bmatrix}
	\end{align*}
for all \(k=1,\dots ,m\ \) with $\ \tilde{\alpha}^{1}_{k,\ell}=(-1)^{\ell}\alpha^{1}_{k,\ell}$ and $\tilde{\alpha}^{2}_{k,\ell}=(-1)^{\ell}\alpha^{2}_{k,\ell}, \ \ell=1,\dots, k,$ with $\alpha^{1}_{k,\ell},\ \alpha^{2}_{k,\ell}$ as in Definition \ref{def_coefficients_c} and Definition \ref{coefficients2}, respectively.

\end{thm}
Note that in the Theorem above we assume the convention that the last sums are not active  if $k=1$.

\smallskip We combine our previous two results into one to give a more general form. Therefore, let \(d=2,3\) be the degree of a Hermite subdivision scheme \(H_{\mathcal{A}}\).
Moreover, let
	\begin{align*}
	\textbf{q}_{2}= \begin{bmatrix} 2q_{k,2\tau}(\frac{-\tau}{2}) \\\tilde{q}_{k,2\tau}(\frac{-\tau}{2})\end{bmatrix} \quad \text{and} \quad \textbf{q}_{3}=\begin{bmatrix} 2q_{k,2\tau}(\frac{-\tau}{2}) \\\tilde{q}_{k,2\tau}(\frac{-\tau}{2}) \\ \frac{1}{2}\hat{q}_{k,2\tau}(\frac{-\tau}{2})\end{bmatrix}.
	\end{align*}

\begin{thm} \label{main_result_2,3}
The Hermite scheme \(H_{\mathcal{A}}\) of order \(d=2, 3\) reproduces constants if and only if

$$
	\textbf{A}(-1)\textbf{e}_{1,d}=\textbf{0}_{d},\quad
	\textbf{A}(1)\textbf{e}_{1,d}=2\textbf{e}_{1,d}.
	$$
	
Moreover, \(H_{\mathcal{A}}\) reproduces polynomials up to degree \(m \geqslant 1\) if and only if it reproduces constants and 
	\begin{align*}
	\textbf{A}^{(k)}(-1) \textbf{e}_{1,d}+\sum_{s=2}^{d}\Big(\sum_{\ell =s-1}^k \alpha^{d}_{k,\ell}\cdot \textbf{A}^{(k-\ell)}(-1) \textbf{e}_{s,d}\Big)&=\textbf{0}_{d},\\
\textbf{A}^{(k)}(1) \textbf{e}_{1,d}+\sum_{s=2}^{d}\Big(\sum_{\ell =s-1}^k \tilde{\alpha}^{d}_{k,\ell}\cdot \textbf{A}^{(k-\ell)}(1) \textbf{e}_{s,d}\Big)&=\textbf{q}_{d}.
	\end{align*}
for all \(k=1,\dots ,m\) with \(\tilde{\alpha}^{1}_{k,\ell}=(-1)^{\ell}\alpha^{1}_{k,\ell}\) and \(\tilde{\alpha}^{2}_{k,\ell}=(-1)^{\ell}\alpha^{2}_{k,\ell},\ \ell=1,\dots,k\) 
with $\alpha^{1}_{k,\ell},\ \alpha^{2}_{k,\ell}$ and  as in Definition \ref{def_coefficients_c} and Definition \ref{coefficients2}, respectively.
\end{thm}

Theorem \ref{main_result_2,3} is the one we plan to extend to Hermite subdivision schemes of any degree \(d\geqslant 3\).


\subsection{Example of interpolatory scheme of order $d=3$}
Consider the primal and interpolatory Hermite scheme studied in \cite{conti1}. The non-zero matrices of its mask are given by
	\begin{align*}
	A_{-1}=\textbf{D}\begin{bmatrix} \lambda_{1} & \lambda_{2}& \lambda_{3}\\ \mu_{1} & \mu_{2}& \mu_{3}\\ \epsilon_{1} & \epsilon_{2}& \epsilon_{3}\\
 	\end{bmatrix}, \quad A_{0}=\textbf{D}, \quad A_{1}=\textbf{D}\begin{bmatrix} \lambda_{1} & -\lambda_{2}& \lambda_{3}\\ -\mu_{1} & \mu_{2}& -\mu_{3}\\ \epsilon_{1} & -\epsilon_{2}& \epsilon_{3}\\
 	\end{bmatrix}
 	\end{align*}
with \(\textbf{D}=diag(1, \frac{1}{2}, \frac{1}{4})\) and parameters \(\lambda_i, \mu_i, \epsilon_i \in \mathbb{R}\).
It i known that the scheme reproduces polynomials up to degree 3 if 
	\begin{align} \label{Equation3_1}
\lambda_{1}=\frac{1}{2}, ~~~ \epsilon_{1}=0, ~~~ \mu_{2}=\frac{1-\mu_{1}}{2},~~~ \epsilon_{3}=\frac{1-\epsilon_{2}}{2},~~ \lambda_{3}=\frac{-1-8\lambda_{2}}{16},~~ \mu_{3}=\frac{2\mu_{1}-3}{24}.
	\end{align}
Computations show that this coincides with our algebraic conditions given by Theorem \ref{main_result_3}. 

In particular, the first two relations of (\ref{Equation3_1}) are necessary for the reproduction of constants.

In order to reproduce linear polynomials the scheme has to satisfy \(\mu_{2}=\frac{1-\mu_{1}}{2}\). Quadratic reproduction leads to \(\epsilon_{3}=\frac{1-\epsilon_{2}}{2}\) and \(\lambda_{3}=\frac{-1-8\lambda_{2}}{16}\). The condition \(\mu_{3}=\frac{2\mu_{2}-3}{24}\) belongs to the reproduction of polynomials of degree 3.

\subsection*{Acknowledgements}
This work was initiated while the  second author was visiting DIEF of the University of Florence. Support from the Austrian Science Fund (FWF): W1230 is gratefully acknowledged.
This research has been accomplished within RITA (Research Italian network on Approximation).

\end{document}